\documentclass[11pt,a4paper]{article}
\usepackage[utf8]{inputenc}
\usepackage{amsmath,amssymb,amsfonts,amsthm}
\usepackage{physics}
\usepackage{graphicx}
\usepackage{hyperref}
\usepackage{xcolor}
\usepackage{geometry}
\usepackage{booktabs}
\numberwithin{equation}{section}
\usepackage{lipsum}

\theoremstyle{plain}
\newtheorem{thm}{Theorem}[section]

\newtheorem{defn}[thm]{Definition}
\theoremstyle{remark}
\newtheorem{rmk}[thm]{Remark}

\title{\textbf{Formal Stability of Tetrahedral Non-Zonal Flows on the Sphere}}
\author{Yuri Cacchiò\\
\scriptsize{Faculty of Mathematics, University of Vienna, Oskar-Morgenstern-Platz 1, 1090 Vienna, Austria.}\\ \href{mailto:yuri.cacchio@univie.ac.at}{\scriptsize{yuri.cacchio@univie.ac.at}}}
\date{\today}

\begin{document}

\maketitle

\begin{abstract}
We investigate the formal stability of finite-amplitude non-zonal flows bifurcating from the trivial state in the unforced 2D Euler equations on the sphere. To bypass the degeneracy of the spherical Laplacian and filter out the low-frequency Fj{\o}rtoft instabilities, we restrict the functional space to the invariant subspace of the tetrahedral symmetry group. Using Arnold's Energy-Casimir method, we prove that the linearized elliptic operator derived via Liapunov-Schmidt reduction acts as the Hessian of the conserved functional. By tracking the critical eigenvalue along the bifurcating branches via the Crandall-Rabinowitz theorem, we establish a relation between the bifurcation topology and formal stability. Applying this framework to four distinct geophysical profile functions, we demonstrate that subcritical polynomial and supercritical sine-Gordon flows achieve a negative-definite second variation, that is, their formal stability. In contrast, subcritical sinh-Gordon and supercritical Liouville exponential flows generate saddle points, making them unstable. This classification identifies the specific nonlinear interactions required for the persistence of large-scale coherent waves in planetary atmospheres.

\vspace{0.5cm}
\end{abstract}

\tableofcontents

\section{Introduction}

The formation and propagation of large-scale coherent structures in planetary atmospheres are modeled by the incompressible 2D Euler equations on the rotating sphere $\mathbb{S}^2$ \cite{constantin2021modelling,holton2013introduction,taylor2016euler,vallis2017atmospheric}. Let the physical domain be parameterized by longitude $\varphi \in [0, 2\pi)$ and latitude $\theta \in \left(-\frac{\pi}{2},\frac{\pi}{2}\right)$, with the polar axis denoted by $x_3 = \sin\theta$. In a rotating frame with constant angular velocity $\omega$, the evolution of the stream function $\psi(t, \varphi, \theta)$ is governed by the conservation of vorticity \cite{majda2002vorticity},
\begin{equation}
    \partial_t (\Delta\psi) + J(\psi, \Delta\psi + 2\omega x_3) = 0.
\end{equation}
Here, the linear term $2\omega x_3$ represents the Coriolis parameter. The relative scalar vorticity is given by the spherical Laplace-Beltrami operator, written as
\begin{equation}
    \Delta\psi = \frac{1}{\cos\theta} \frac{\partial}{\partial\theta} \left( \cos\theta \frac{\partial\psi}{\partial\theta} \right) + \frac{1}{\cos^2\theta} \frac{\partial^2\psi}{\partial\varphi^2}
\end{equation}
and the spherical Jacobian is defined as 
\begin{equation}
    J(f,g) = \frac{1}{\cos\theta} (\partial_\varphi f \partial_\theta g - \partial_\theta f \partial_\varphi g).
\end{equation}
Steady geophysical flows, such as atmospheric jet streams, are characterized by stationary condition $\partial_t = 0$. This assumption requires the stream function and the vorticity to commute, 
\begin{equation}
    J(\psi, \Delta\psi + 2\omega x_3) = 0.
\end{equation}
Following Arnold's argument \cite{arnold1998} and the rigorous formulations of Constantin and Johnson \cite{constantin2017large}, the absence of Jacobian implies that, wherever $\nabla\psi \neq 0$, the absolute vorticity is a functional of the stream function. This relation ensures the existence of a sufficiently smooth function $F$ satisfying the nonlinear elliptic equation
\begin{equation} \label{eq:CG_euler}
    -\Delta\psi + F(\psi) = 2\omega x_3.
\end{equation}
The simplest solutions to \eqref{eq:CG_euler} are \textit{zonal flows}, which depend only on the latitudinal coordinate $x_3$ \cite{pedlosky1987geophysical,vallis2017atmospheric}. However, at critical parameter thresholds, these symmetric states can undergo spontaneous symmetry breaking, generating two-dimensional, non-zonal \textit{Rossby-Haurwitz waves} \cite{haurwitz1940motion,rossby1939relations} via local bifurcation. The classical stability of these planetary waves has been investigated through numerical and spectral methods \cite{baines1976stability, hoskins1973stability}.

Analyzing this symmetry-breaking mechanism directly in the rotating frame is challenging because of the asymmetry introduced by the Coriolis force. To isolate the nature of the bifurcation, we adopt the reduction strategy of Constantin and Germain \cite{constantin2022} by passing to the non-rotating limit ($\omega = 0$). In this regime, the problem simplifies to the unforced semilinear equation
\begin{equation}\label{eq:stationary_euler}
    -\Delta\psi + F(\lambda, \psi) = 0 \quad \text{on } \mathbb{S}^2,
\end{equation}
where the non-linearity is parameterized by the bifurcation variable $\lambda \in \mathbb{R}$, under the trivial-branch condition $F(\lambda, 0) = 0$. 

Notice that this irrotational simplification does not compromise physical relevance. Constantin and Germain \cite{constantin2022} established that any stationary solution $\psi_0(\varphi, \theta)$ of the non-rotating problem \eqref{eq:stationary_euler} corresponds to a non-zonal traveling wave on the rotating sphere via the transformation 
\begin{equation}
    \psi_\omega(\varphi, \theta, t) = \psi_0(\varphi + \omega t, \theta) + \omega x_3.
\end{equation}
This mapping ensures that our analysis of stationary states at $\omega=0$ characterizes planetary waves propagating westward with a phase velocity $c = \omega$ relative to the planetary surface.

A central problem in this framework is the stability of the finite-amplitude non-zonal flows emerging from \eqref{eq:stationary_euler}. It is a well-known consequence of Fj{\o}rtoft's theorem \cite{fjortoft1953changes, vallis2017atmospheric} and Arnold's classical nonlinear stability criteria \cite{arnold1969apriori,arnold1998} that stationary flows dominated by spherical harmonics of degree $l \ge 2$ act as saddle points in the energy-enstrophy invariant manifold. Consequently, these high-frequency structures are unstable to large-scale, low-frequency perturbations.

In a previous work \cite{yuri2026}, we investigated the emergence of these solutions from the trivial state. To bypass the $(2l+1)$-dimensional kernel degeneracy of the spherical Laplacian, which prohibits the direct application of classical local bifurcation theorems \cite{crandall1971bifurcation, kielhofer2012bifurcation}, we restricted the functional space to the invariant subspace of the tetrahedral symmetry group $\mathbf{T}$ using equivariant bifurcation theory \cite{golubitsky1988singularities}. This constraint filters out the large-scale asymmetric modes and collapses the null space to a single dimension spanned by the $l=3$ real tetrahedral harmonic $Y^*$. By applying a Liapunov-Schmidt reduction, we demonstrated that the topology of the resulting bifurcating branches, $\psi(\varepsilon) = \varepsilon Y^* + \mathcal{O}(\varepsilon^2)$, is strictly determined by the chosen non-linearity $F(\lambda, \psi)$ and the mass constraint. Specifically, we proved that generalized polynomial and sinh-Gordon models bifurcate subcritically, whereas sine-Gordon and Liouville exponential models bifurcate supercritically.

Having established the existence and the geometry of these tetrahedral non-zonal flows, the aim of the present paper is to evaluate their robustness against perturbations, a property we define as \textit{tetrahedral formal stability}. Because the geometric restriction to the tetrahedral subgroup projects out the unstable low-frequency modes $l=1, 2$, it provides a mathematical setting to test the stability of the $l=3$ bifurcating branches without the onset of Fj{\o}rtoft's instability.

Since the 2D Euler equations describe a conservative Hamiltonian system \cite{arnold1989mathematical}, stability cannot be established by treating the dynamics as a standard gradient flow. Instead, we evaluated it through the second variation of the Arnold Energy-Casimir functional. In this paper, we prove that, when restricted to the invariant subspace $X_{\mathbf{T}}$, the linearized elliptic operator resulting from the Liapunov-Schmidt reduction acts as the second variation of the Energy-Casimir functional. Employing the Crandall-Rabinowitz theorem \cite{crandall1973stability}, we trace the critical eigenvalue $\mu(\varepsilon)$ of the differential operator along the bifurcating curves. 

This spectral link allows us to relate the bifurcation direction to the definiteness of the restricted Energy-Casimir functional. By evaluating the polynomial, sine-Gordon, sinh-Gordon, and exponential tetrahedral flows, we demonstrate that formal stability is not governed by the bifurcation topology alone. Our analysis reveals that both subcritical and supercritical branches can achieve formal stability, depending on the nonlinear interactions governing the perturbation of the critical eigenvalue.

The organization of this paper is as follows. In Section \ref{sec:preliminaries}, we introduce the Arnold Energy-Casimir functional, define the criteria for formal stability on the restricted tetrahedral subspace, and state the main stability theorem. Section \ref{sec:proof} is devoted to the rigorous proof of the main result, utilizing a perturbative spectral expansion to link the geometric Hessian to the Crandall-Rabinowitz bifurcation framework. In Section \ref{sec:physical_models}, we apply this criterion to evaluate the formal stability of the four geophysical profile functions. Finally, Section \ref{sec:conclusions} provides a summary and discusses the physical implications of these stability classifications.

\section{Preliminaries and Main Result}\label{sec:preliminaries}

\subsection{Energy-Casimir Functional and Formal Stability}

In dissipative systems, the derivative of the stationary problem often acts as the linearized operator, making its spectrum sufficient to determine asymptotic linear stability. However, the 2D Euler equations describe a conservative, Hamiltonian fluid. The linearized operator does not correspond to the elliptic derivative $L$ emerging from the Liapunov-Schmidt reduction
\begin{equation}
    L = -\Delta + \partial_\psi F(\lambda, \psi_0)I.
\end{equation}
Consequently, the spectrum of $L$ alone does not yield information about perturbation growth. 

To link the spatial operator $L$ with stability, we employ Arnold's Energy-Casimir method. Specifically, we demonstrate that $L$ acts as the Hessian (the second variation) of the Energy-Casimir functional.

The non-rotating 2D Euler equations conserve the total kinetic energy,
\begin{equation}
    E[\psi] = \frac{1}{2} \iint_{\mathbb{S}^2} |\nabla \psi|^2 \, d\sigma = \frac{1}{2} \iint_{\mathbb{S}^2} \psi (-\Delta \psi) \, d\sigma,
\end{equation}
and an infinite family of Casimir invariants (generalized enstrophies) \cite{arnold1998},
\begin{equation}
    C[\psi] = \iint_{\mathbb{S}^2} \Phi(-\Delta \psi) \, d\sigma,
\end{equation}
where $\Phi$ is an arbitrary, sufficiently smooth function and $-\Delta \psi$ is the scalar vorticity.

Arnold's stability theorem relies on constructing the Energy-Casimir functional,
\begin{equation}\label{eq:H_C}
    H_C[\psi] = E[\psi] + C[\psi].
\end{equation}
A stationary flow $\psi_0$ is a formal relative equilibrium if it is a critical point of this functional, meaning its first variation vanishes $\delta H_C[\psi_0] = 0$. Computing the first variation with respect to an arbitrary perturbation $\delta \psi$ yields
\begin{equation}
    \delta H_C = \iint_{\mathbb{S}^2} \left( \Phi'(-\Delta \psi_0) + \psi_0 \right) (-\Delta \delta \psi) \, d\sigma = 0.
\end{equation}
For this condition to hold for any generic perturbation, we require
\begin{equation}
    \Phi'(-\Delta \psi_0) = -\psi_0.
\end{equation}
Recalling the stationary Euler equation \eqref{eq:stationary_euler}, we substitute this into the identity to obtain
\begin{equation}\label{eq:phi'}
    \Phi'(-F(\lambda, \psi_0)) = -\psi_0.
\end{equation}
To establish formal stability, we examine the second variation at the critical point $\psi_0$,
\begin{equation}
    \delta^2 H_C[\delta \psi] = \frac{1}{2} \iint_{\mathbb{S}^2} \left[ |\nabla \delta \psi|^2 + \Phi''(-\Delta \psi_0) (-\Delta \delta \psi)^2 \right] \, d\sigma.
\end{equation}
Differentiating \eqref{eq:phi'} with respect to $\psi_0$, we derive 
\begin{equation}
    \Phi''(-\Delta \psi_0) = \frac{1}{\partial_\psi F(\lambda, \psi_0)}.
\end{equation}
Substituting this into the second variation gives
\begin{equation}
    \delta^2 H_C[\delta \psi] = \frac{1}{2} \iint_{\mathbb{S}^2} \left[ |\nabla \delta \psi|^2 + \frac{1}{\partial_\psi F(\lambda, \psi_0)} (-\Delta \delta \psi)^2 \right] \, d\sigma.
\end{equation}
We rewrite the kinetic energy term using Green's first identity. The second variation becomes
\begin{equation}\label{eq:second_variation}
    \delta^2 H_C[\delta \psi] = \frac{1}{2} \iint_{\mathbb{S}^2} \left[ \delta \psi (-\Delta \delta \psi) + \frac{1}{\partial_\psi F(\lambda, \psi_0)} (-\Delta \delta \psi)^2 \right] \, d\sigma.
\end{equation}
We recover $L = -\Delta + \partial_\psi F(\lambda, \psi_0) I$ inside the integral,
\begin{align}\label{eq:second_variation_L}
    \delta^2 H_C[\delta \psi] &= \frac{1}{2} \iint_{\mathbb{S}^2} \frac{(-\Delta \delta \psi)}{\partial_\psi F(\lambda, \psi_0)} \Big( \partial_\psi F(\lambda, \psi_0) \delta \psi - \Delta \delta \psi \Big) \, d\sigma \nonumber \\
    &= \frac{1}{2} \iint_{\mathbb{S}^2} \frac{(-\Delta \delta \psi)}{\partial_\psi F(\lambda, \psi_0)} L (\delta \psi) \, d\sigma.
\end{align}
To study the definiteness of the second variation, we express it as an inner product,
\begin{equation}
    \delta^2 H_C[\delta \psi] = \frac{1}{2} \langle \delta \psi, \mathcal{H} \delta \psi \rangle_{L^2}.
\end{equation}
Isolating $\delta \psi$ in \eqref{eq:second_variation}, we define the Hessian as the following fourth-order, self-adjoint elliptic operator
\begin{equation} \label{eq:hessian_operator}
    \mathcal{H} = -\Delta + \Delta \left( \frac{1}{\partial_\psi F(\lambda, \psi_0)} \Delta \right).
\end{equation}
When evaluated along the bifurcating branch, where $\lambda = \lambda(\varepsilon)$ and $\psi_0 = \psi(\varepsilon)$, we denote this operator as $\mathcal{H}(\varepsilon)$.

\begin{defn}[Tetrahedral Formal Stability] \label{def:tetra_stability}
Let $X_{\mathbf{T}} \subset L^2(\mathbb{S}^2)$ denote the invariant subspace of tetrahedral symmetric functions. A bifurcating steady state $\psi(\varepsilon) \in X_{\mathbf{T}}$ is said to be \textit{formally stable} in the tetrahedral sense if the Hessian operator $\mathcal{H}(\varepsilon)$ of the Energy-Casimir functional is negative-definite over the perturbation space $X_{\mathbf{T}}$. 
\end{defn}

\begin{rmk}[Formal and Nonlinear Stability]
It is well established that the formal stability of infinite-dimensional Hamiltonian systems, defined by the definiteness of the second variation $\delta^2 H_C$, does not imply nonlinear (Lyapunov) stability due to the lack of general coercivity bounds. The classical criteria for nonlinear stability are provided by Arnold's theorems \cite{arnold1969apriori,arnold1998}.
The Second Arnold Theorem requires $0 > \partial_\psi F > -\lambda_1$, where $\lambda_1 = 2$ is the first non-zero eigenvalue of the spherical Laplacian. Since $\partial_\psi F = -12$ near the bifurcation, this general condition fails.

However, by restricting our functional space to the tetrahedral invariant subspace $X_{\mathbf{T}}$, we filter out the low-frequency modes $l \in \{1, 2\}$, shifting the spectral lower bound of the Laplacian to $\lambda_{\text{min}} = 12$. Consequently, the Second Arnold criterion on this restricted subspace requires $0 > \partial_\psi F > -12$. Because our critical bifurcating state originates at this singularity limit ($\partial_\psi F(\lambda^*, 0) = -12$), proving nonlinear stability requires higher-order coercivity bounds \cite{holm1985nonlinear}. 
\end{rmk}

Let us investigate the admissible perturbations $\delta\psi \in X_{\mathbf{T}}$ to evaluate the definiteness of the second variation. Since the velocity field is defined by the gradient of the stream function, a constant perturbation (the $l=0$ harmonic) produces zero velocity and does not carry kinetic energy. We eliminate this trivial mode by restricting our perturbations to zero-mean functions. Furthermore, high-frequency flows are unstable to large-scale perturbations, specifically the $l=1$ and $l=2$ spherical harmonics (Fj{\o}rtoft's instability). However, by reducing our problem to the tetrahedral subspace $X_{\mathbf{T}}$, we forbid these large-scale modes. 

Therefore, by combining the zero-mean condition with tetrahedral symmetry, all low-frequency perturbations $l \in \{0, 1, 2\}$ are projected out. Within this restricted space, the spectrum of the linearized operator $L$ begins at the first available tetrahedral harmonic, which is $l=3$.

The formal stability of the perturbed branch relies on this $l=3$ mode. To see this, we evaluate the unperturbed spectrum of $\mathcal{H}(0)$ at the bifurcation point. The eigenvalues of the Hessian are given by 
\begin{equation}
    \eta_l^{(0)} = -\frac{l(l+1)}{12}[l(l+1) - 12].
\end{equation}
For all high-frequency modes $l \ge 4$, the Hessian eigenvalues are strictly negative and bounded away from zero (e.g., $\eta_4^{(0)} = -40/3 < 0$). By standard spectral perturbation theory, a sufficiently small perturbation $\mathcal{O}(\varepsilon)$ cannot shift these negative eigenvalues above zero. Consequently, $l=3$ (for which $\eta_3^{(0)} = 0$) is the unique critical mode determining the definiteness of the perturbed Hessian.

\subsection{Statement of the Main Theorem}

To track the stability of this critical mode, let the linearized operator along the trivial branch be defined as 
\begin{equation}
    L_0(\lambda) = -\Delta + \partial_\psi F(\lambda, 0)I.
\end{equation}
Evaluating the action of $L_0(\lambda)$ on the $l=3$ tetrahedral harmonic $Y^*$ yields
\begin{equation}
    L_0(\lambda)Y^* = \left( 12 + \partial_\psi F(\lambda, 0) \right) Y^* = \gamma(\lambda) Y^*,
\end{equation}
where $\gamma(\lambda) = 12 + \partial_\psi F(\lambda, 0)$ is the eigenvalue of the trivial branch in the invariant subspace. At the critical parameter $\lambda = \lambda^*$, the necessary condition for bifurcation requires the kernel to be non-trivial, i.e. setting $\gamma(\lambda^*) = 0$ and $\partial_\psi F(\lambda^*, 0) = -12$.

Because the symmetry restriction $X_{\mathbf{T}}$ isolates $l=3$ as the critical mode, the Crandall-Rabinowitz framework \cite{crandall1971bifurcation,crandall1973stability} guarantees that, under specific transversality conditions (see Theorem \ref{thm:stability} below), this eigenvalue shifts continuously along the bifurcating branch. If it shifts into the negative half-plane, the formation of a saddle point is precluded, ensuring the formal stability of the new non-zonal flow. We formalize this mechanism in the following main theorem.

\begin{thm}[\textbf{Tetrahedral Formal Stability}]\label{thm:stability}
Let $F(\lambda, \psi)$ be a smooth nonlinearity satisfying $F(\lambda, 0) = 0$ and $\partial_\psi F(\lambda^*, 0) = -12$. Assume the nonlinearity possesses a pitchfork bifurcation ($\lambda_1 = 0$), parameterized by $\lambda(\varepsilon) = \lambda^* + \varepsilon^2 \lambda_2 + \mathcal{O}(\varepsilon^3)$.

Let $\psi(\varepsilon) = \varepsilon Y^* + \varepsilon^2 \psi_2 + \mathcal{O}(\varepsilon^3)$ be a finite-amplitude tetrahedral flow bifurcating from the trivial state at $\lambda = \lambda^*$ in the invariant subspace $X_{\mathbf{T}}$. Let $\gamma(\lambda)$ denote the simple eigenvalue of the linearized operator along the trivial branch.

The bifurcating flow $\psi(\varepsilon)$ is tetrahedral formally stable for sufficiently small $\varepsilon \neq 0$ if and only if the transversality condition $\gamma'(\lambda^*)$ and the bifurcation parameter $\lambda_2$ possess opposite signs,
\begin{equation}
    \lambda_2 \cdot \gamma'(\lambda^*) < 0.
\end{equation}
Conversely, if $\lambda_2 \cdot \gamma'(\lambda^*) > 0$, the flow generates a saddle point and is formally unstable.
\end{thm}

The proof of this theorem, based on an expansion of the Energy-Casimir Hessian, is detailed in Section \ref{sec:proof}. Subsequently, in Section \ref{sec:physical_models}, we apply this general criterion to four specific geophysical profile functions, using the bifurcation topologies derived in \cite{yuri2026}.

\section{Proof of the Main Result}\label{sec:proof}

In this section, we prove the main result stated in Theorem \ref{thm:stability}. 
Our strategy relies on a perturbative expansion of the fourth-order Hessian $\mathcal{H}(\varepsilon)$. The proof follows three main steps:
\begin{itemize}
    \item First, we analyze the unperturbed spectrum to show that the definiteness relies only on the $l=3$ critical eigenvalue $\eta(\varepsilon)$.
    \item Second, we perform a standard perturbation expansion of both the linearized operator $L(\varepsilon)$ and the geometric Hessian $\mathcal{H}(\varepsilon)$ to extract their leading-order eigenvalue corrections, $\mu_2$ and $\eta_2$ respectively.
    \item Third, by exploiting the Fredholm alternative and some cancellations on the unit sphere, we prove the identity $\eta_2 = -\mu_2$, which allows us to close the stability problem using the Crandall-Rabinowitz bifurcation framework.
\end{itemize}

\begin{proof}

We investigate the formal stability of the tetrahedral flow by evaluating the spectrum of the Energy-Casimir second variation $\delta^2 H_C$ over $X_{\mathbf{T}}$.
To evaluate this, we recall definition \eqref{eq:hessian_operator}
\begin{equation}
    \mathcal{H}(\varepsilon) = -\Delta + \Delta \left( [\partial_\psi F]^{-1} \Delta \right).
\end{equation}
Since we have $\partial_\psi F(\lambda^*, 0) = -12 < 0$, the pre-factor $[\partial_\psi F(\lambda(\varepsilon), \psi(\varepsilon))]^{-1}$ remains smooth, bounded, and strictly negative for all sufficiently small $\varepsilon$. This regularity ensures that the functional $\mathcal{H}(\varepsilon)$ is well-defined in the neighborhood of the bifurcation.

Let us evaluate the unperturbed Hessian $\mathcal{H}(0)$ along the trivial branch $(\lambda^*, 0)$. Since $\partial_\psi F(\lambda^*, 0) = -12$, we obtain
\begin{equation}
    \mathcal{H}(0) = -\Delta - \frac{1}{12}\Delta^2 = \frac{1}{12}\Delta \big( -12 - \Delta \big) = \frac{1}{12}\Delta L_0.
\end{equation}
We now analyze the spectrum of $\mathcal{H}(0)$ restricted to the invariant subspace $X_{\mathbf{T}}$:
\begin{enumerate}
    \item For $l=3$, the linearized operator $L_0$ vanishes, $L_0 Y^* = 0$. Hence, $\mathcal{H}(0)$ has a zero eigenvalue with a one-dimensional kernel spanned by $Y^*$.
    \item For all higher-order harmonics $l \ge 4$, the eigenvalues of $L_0$ are strictly positive (e.g., $L_0 Y_4 = (20-12)Y_4 = 8Y_4$). Since the Laplacian $\Delta$ is a strictly negative operator (its spectrum on $X_{\mathbf{T}}$ begins at $-12$ and decreases), the product $\frac{1}{12}\Delta L_0$ forces the spectrum of $\mathcal{H}(0)$ to be strictly negative and bounded away from zero for all $l \ge 4$.
\end{enumerate}

When we perturb the system $\varepsilon \neq 0$, as already mentioned the pre-factor $[\partial_\psi F]^{-1}$ remains smooth and uniformly bounded away from zero. By standard perturbation theory for regular elliptic operators, the spectral branches shift continuously. Because the unperturbed spectrum for $l \ge 4$ is separated from zero by a finite gap, it remains strictly negative for sufficiently small $\varepsilon$.

Consequently, the definiteness of $\mathcal{H}(\varepsilon)$ on $X_{\mathbf{T}}$ is determined only by the sign of the perturbed eigenvalue $\eta(\varepsilon)$ emerging at $l=3$. If $\eta(\varepsilon) < 0$, the Hessian is negative-definite.

Let the potential be expanded along the bifurcating branch as
\begin{equation}
    \partial_\psi F(\lambda(\varepsilon), \psi(\varepsilon)) = -12 + \varepsilon V_1 + \varepsilon^2 V_2 + \mathcal{O}(\varepsilon^3).
\end{equation}
The first-order coefficient $V_1$ is obtained by taking the derivative with respect to $\varepsilon$ and evaluating it at the bifurcation point $\varepsilon = 0$,
\begin{equation}
    V_1 = \left. \frac{d}{d\varepsilon} \Big[ \partial_\psi F(\lambda(\varepsilon), \psi(\varepsilon)) \Big] \right|_{\varepsilon=0} = \partial_{\lambda \psi} F(\lambda^*, 0) \lambda'(0) + \partial_{\psi \psi} F(\lambda^*, 0) \psi'(0).
\end{equation}
Recall that the bifurcating branch emerges from the trivial state, meaning $\psi(0) = 0$, and that $\psi'(0) = Y^*$. Furthermore, due to the symmetry of the bifurcation $\lambda'(0) = \lambda_1 = 0$. Substituting these values into the above relation, yields
\begin{equation}\label{eq:V_1}
    V_1 = \partial_{\psi\psi}F(\lambda^*, 0) Y^*.
\end{equation}
To determine the first-order eigenvalue correction, we consider the perturbed eigenvalue problem for the linearized operator acting on the eigenfunction $w(\varepsilon)$,
\begin{equation} \label{eq:app_eigen_prob}
    L(\varepsilon) w(\varepsilon) = \mu(\varepsilon) w(\varepsilon).
\end{equation}
The operators and variables are expanded in a power series with respect to the parameter $\varepsilon$ as follows,
\begin{align}
    L(\varepsilon) &= L_0 + \varepsilon V_1 + \varepsilon^2 V_2 + \mathcal{O}(\varepsilon^3), \\
    w(\varepsilon) &= Y^* + \varepsilon w_1 + \varepsilon^2 w_2 + \mathcal{O}(\varepsilon^3), \\
    \mu(\varepsilon) &= \varepsilon \mu_1 + \varepsilon^2 \mu_2 + \mathcal{O}(\varepsilon^3),
\end{align}
where we used the bifurcation condition $\mu(0) = 0$ and chose the base eigenfunction $w(0) = Y^* \in \ker(L_0)$, such that $\|Y^*\|^2 = 1$. 

Substituting these expansions into \eqref{eq:app_eigen_prob} and collecting terms by powers of $\varepsilon$, we obtain the following equations.

At order $\mathcal{O}(1)$, we recover 
\begin{equation}
    L_0 Y^* = 0.
\end{equation}

At order $\mathcal{O}(\varepsilon)$, we obtain 
\begin{equation}\label{eq:first_order}
    L_0 w_1 + V_1 Y^* = \mu_1 Y^*.
\end{equation} 
To isolate $\mu_1$, we use the Fredholm alternative by projecting the equation onto the kernel generated by $Y^*$. Taking the $L^2$-inner product with $Y^*$, we get
\begin{equation}
    \langle L_0 w_1, Y^* \rangle + \langle V_1 Y^*, Y^* \rangle = \mu_1 \|Y^*\|^2.
\end{equation}
By the self-adjointness of $L_0$ and recalling that $L_0 Y^* = 0$, the first term vanishes 
\begin{equation}
    \langle L_0 w_1, Y^* \rangle = \langle w_1, L_0 Y^* \rangle = 0.
\end{equation}
Consequently, the first-order correction $\mu_1$ is given by 
\begin{equation}
    \mu_1 = \langle V_1 Y^*, Y^* \rangle \sim \iint_{\mathbb{S}^2} V_1 (Y^*)^2 \, d\sigma.
\end{equation}
From \eqref{eq:V_1} we know that $V_1 \sim Y^*$, then the integrand is proportional to $(Y^*)^3$, which is odd. Integrating over $\mathbb{S}^2$ yields  zero. Thus, we obtain
\begin{equation}\label{eq:mu_1}
    \mu_1 = 0.
\end{equation}
Therefore, from \eqref{eq:first_order} we obtain
\begin{equation} \label{eq:app_order1}
     L_0 w_1 = -V_1 Y^*.
\end{equation}
Because the term $V_1 Y^*$ is orthogonal to the kernel $\langle V_1 Y^*, Y^* \rangle = 0$, equation \eqref{eq:app_order1} is solvable. We determine the first-order eigenfunction correction by applying the pseudo-inverse $S_L = (L_0|_{\ker(L_0)^\perp})^{-1}$ and imposing the orthogonality condition $\langle w_1, Y^* \rangle = 0$:
\begin{equation} \label{eq:app_w1}
    w_1 = -S_L V_1 Y^*.
\end{equation}
At order $\mathcal{O}(\varepsilon^2)$, we have
\begin{equation} \label{eq:app_order2}
    L_0 w_2 + V_1 w_1 + V_2 Y^* = \mu_2 Y^* + \mu_1 w_1.
\end{equation}
Setting $\mu_1 = 0$ and projecting equation \eqref{eq:app_order2} onto the kernel spanned by $Y^*$, we obtain
\begin{equation}
    \langle L_0 w_2, Y^* \rangle + \langle V_1 w_1, Y^* \rangle + \langle V_2 Y^*, Y^* \rangle = \mu_2 \langle Y^*, Y^* \rangle.
\end{equation}
By the self-adjointness of $L_0$, the first term vanishes $\langle L_0 w_2, Y^* \rangle = 0$. Since $\|Y^*\|^2 = 1$, we derive
\begin{equation}
    \mu_2 = \langle V_2 Y^*, Y^* \rangle + \langle V_1 w_1, Y^* \rangle.
\end{equation}
Finally, we substitute \eqref{eq:app_w1}. Exploiting the fact that the potential $V_1$ acts as a real, self-adjoint multiplicative operator, we can move it across the inner product,
\begin{align}\label{eq:mu2_app}
    \mu_2 &= \langle V_2 Y^*, Y^* \rangle + \langle V_1 (-S_L V_1 Y^*), Y^* \rangle \nonumber \\
    &= \langle V_2 Y^*, Y^* \rangle - \langle S_L V_1 Y^*, V_1 Y^* \rangle \nonumber \\
    &= \langle V_2 Y^*, Y^* \rangle - \langle V_1 Y^*, S_L V_1 Y^* \rangle.
\end{align}
Let us now proceed similarly for $\mathcal{H}(\varepsilon)$.
We expand the pre-factor of the Hessian using the Taylor series $(1-x)^{-1} = 1 + x + x^2 + \dots$
\begin{align}
    [\partial_\psi F]^{-1} &= \frac{1}{-12 \left( 1 - \frac{\varepsilon V_1 + \varepsilon^2 V_2}{12} \right)} \nonumber \\
    &= -\frac{1}{12} \left[ 1 + \left( \frac{\varepsilon V_1 + \varepsilon^2 V_2}{12} \right) + \left( \frac{\varepsilon V_1}{12} \right)^2 \right] + \mathcal{O}(\varepsilon^3) \nonumber \\
    &= -\frac{1}{12} - \frac{\varepsilon V_1}{144} - \varepsilon^2 \left( \frac{V_2}{144} + \frac{V_1^2}{1728} \right) + \mathcal{O}(\varepsilon^3).
\end{align}
Isolating the operator orders $\mathcal{H}(\varepsilon) = \mathcal{H}_0 + \varepsilon \mathcal{H}_1 + \varepsilon^2 \mathcal{H}_2$, we identify
\begin{align}
    \mathcal{H}_0 &= -\Delta - \frac{1}{12}\Delta^2, \\
    \mathcal{H}_1 &= -\frac{1}{144}\Delta (V_1 \Delta), \label{eq:def_H_1}\\
    \mathcal{H}_2 &= -\frac{1}{144}\Delta \left( \left(V_2 + \frac{1}{12}V_1^2 \right) \Delta \right).\label{eq:def_H_2}
\end{align}
We define the following eigenvalue problem
\begin{equation}\label{eq:eigenvalue_problem_hess}
    \mathcal{H}(\varepsilon) Y(\varepsilon) = \eta(\varepsilon) Y(\varepsilon).
\end{equation}
We expand the operator, the eigenvalue, and the eigenvector in powers of $\varepsilon$,
\begin{align}
    \mathcal{H}(\varepsilon) &= \mathcal{H}_0 + \varepsilon \mathcal{H}_1 + \varepsilon^2 \mathcal{H}_2 + \mathcal{O}(\varepsilon^3), \\
    \eta(\varepsilon) &= \varepsilon \eta_1 + \varepsilon^2 \eta_2 + \mathcal{O}(\varepsilon^3), \\
    Y(\varepsilon) &= Y^* + \varepsilon Y_1 + \varepsilon^2 Y_2 + \mathcal{O}(\varepsilon^3),
\end{align}
where we used the fact that the unperturbed eigenvalue is $\eta_0 = 0$. Substituting these expansions into \eqref{eq:eigenvalue_problem_hess} and collecting terms order by order yields the following equations.

At order $\mathcal{O}(1)$:
\begin{equation}
    \mathcal{H}_0 Y^* = 0,
\end{equation}
which recovers the kernel of the unperturbed operator.

At order $\mathcal{O}(\varepsilon)$:
\begin{equation}\label{eq:order_1_H}
    \mathcal{H}_0 Y_1 + \mathcal{H}_1 Y^* = \eta_1 Y^*.
\end{equation}
Taking the inner product with $Y^*$ and exploiting the self-adjointness of $\mathcal{H}_0$, the first term vanishes $\langle \mathcal{H}_0 Y_1, Y^* \rangle = 0$. Since $\|Y^*\|^2 = 1$, we obtain
\begin{equation}
    \eta_1 = \langle \mathcal{H}_1 Y^*, Y^* \rangle.
\end{equation}
Recalling definition \eqref{eq:def_H_1}, we evaluate this product by exploiting the self-adjointness of the Laplacian and the relation $\Delta Y^* = -12Y^*$,
\begin{align}
    \eta_1 &= -\frac{1}{144} \langle \Delta(V_1 \Delta Y^*), Y^* \rangle \nonumber \\
           &= -\frac{1}{144} \langle V_1 \Delta Y^*, \Delta Y^* \rangle \nonumber \\
           &= -\frac{1}{144} \langle V_1 (-12 Y^*), -12 Y^* \rangle \nonumber \\
           &= -\langle V_1 Y^*, Y^* \rangle.
\end{align}
As established in \eqref{eq:mu_1}, the odd symmetry of $Y^*$ causes this integral to vanish. Thus, we obtain
\begin{equation}
    \eta_1 = 0.
\end{equation}
Therefore, equation \eqref{eq:order_1_H} simplifies to 
\begin{equation}
    \mathcal{H}_0 Y_1 = -\mathcal{H}_1 Y^*.
\end{equation}
Inverting $\mathcal{H}_0$ on the orthogonal complement of its kernel via the pseudo-inverse $S_{\mathcal{H}}= (\mathcal{H}_0|_{\ker(\mathcal{H}_0)^\perp})^{-1}$, we find the first-order correction to the eigenvector,
\begin{equation}\label{eq:Y_1}
    Y_1 = -S_{\mathcal{H}} \mathcal{H}_1 Y^*.
\end{equation}

At order $\mathcal{O}(\varepsilon^2)$:
\begin{equation}
    \mathcal{H}_0 Y_2 + \mathcal{H}_1 Y_1 + \mathcal{H}_2 Y^* = \eta_2 Y^* + \eta_1 Y_1.
\end{equation}
Since $\eta_1 = 0$, projecting this equation onto $Y^*$ and again using $\langle \mathcal{H}_0 Y_2, Y^* \rangle = 0$ yields
\begin{equation}
    \eta_2 = \langle \mathcal{H}_2 Y^*, Y^* \rangle + \langle \mathcal{H}_1 Y_1, Y^* \rangle.
\end{equation}
Using the self-adjointness of $\mathcal{H}_1$, we rewrite the second term as $\langle Y_1, \mathcal{H}_1 Y^* \rangle$. Substituting \eqref{eq:Y_1}, we have
\begin{equation}
    \eta_2 = \langle \mathcal{H}_2 Y^*, Y^* \rangle - \langle S_{\mathcal{H}} \mathcal{H}_1 Y^*, \mathcal{H}_1 Y^* \rangle.
\end{equation}
For the first term, using \eqref{eq:def_H_2} and observing that $\Delta Y^* = -12Y^*$, we obtain
\begin{equation} \label{eq:eta2_term1}
    \langle \mathcal{H}_2 Y^*, Y^* \rangle = -\frac{1}{144} \left\langle \left(V_2 + \frac{1}{12}V_1^2\right)(-12Y^*), -12Y^* \right\rangle = -\langle V_2 Y^*, Y^* \rangle - \frac{1}{12}\langle V_1^2 Y^*, Y^* \rangle.
\end{equation}
For the second term, we first compute the action of $\mathcal{H}_1$ on $Y^*$,
\begin{equation}
    \mathcal{H}_1 Y^* = -\frac{1}{144}\Delta(V_1(-12Y^*)) = \frac{1}{12}\Delta(V_1 Y^*).
\end{equation}
To isolate the action of the unperturbed Hessian, we factorize its operator as 
\begin{equation}
    \mathcal{H}_0 = -\Delta - \frac{1}{12}\Delta^2 = \frac{1}{12}\Delta(-\Delta - 12) = \frac{1}{12}\Delta L_0.
\end{equation}
Since the differential operators commute, the pseudo-inverse factorizes as $S_{\mathcal{H}} = 12 S_L \Delta^{-1}$. This operator is well-defined on the range of the Laplacian (i.e., orthogonal to the mode $l=0$). Substituting into the inner product yields
\begin{equation}\label{eq:second_term}
    \langle \mathcal{H}_1 Y^*, S_{\mathcal{H}} \mathcal{H}_1 Y^* \rangle = \left\langle \frac{1}{12}\Delta(V_1 Y^*), 12 S_L \Delta^{-1} \left(\frac{1}{12}\Delta(V_1 Y^*)\right) \right\rangle.
\end{equation}
Notice that the pseudo-inverse recovers the function up to its zero-mean projection,
\begin{equation}
    \Delta^{-1}\Delta(V_1 Y^*) = V_1 Y^* - c_0 Y_0^0,
\end{equation}
where $c_0 Y_0^0$ represents the non-zero mean (the $l=0$ component) of $V_1 Y^* \sim (Y^*)^2$. Substituting this decomposition into the inner product, we obtain
\begin{equation}
    \frac{1}{12} \langle \Delta(V_1 Y^*), S_L (V_1 Y^* - c_0 Y_0^0) \rangle = \frac{1}{12} \langle \Delta(V_1 Y^*), S_L (V_1 Y^*) \rangle - \frac{1}{12} \langle \Delta(V_1 Y^*), S_L (c_0 Y_0^0) \rangle.
\end{equation}
The second term on the right-hand side vanishes. Indeed, by the Divergence Theorem, the integral of a Laplacian over $\partial \mathbb{S}^2 = \emptyset$ is zero,
\begin{equation}
    \iint_{\mathbb{S}^2} \Delta (V_1 Y^*) \, d\sigma = \iint_{\mathbb{S}^2} \nabla \cdot \nabla (V_1 Y^*) \, d\sigma = \oint_{\partial \mathbb{S}^2} \nabla (V_1 Y^*) \cdot \mathbf{n} \, dl = 0.
\end{equation}
This establishes that $\Delta(V_1 Y^*)$ is orthogonal to any constant mode. Since $S_L$ preserves this orthogonality, the term $\langle \Delta(V_1 Y^*), S_L (c_0 Y_0^0) \rangle$ is an inner product between a zero-mean function and a constant, which evaluates to zero. Thus, equation \eqref{eq:second_term} becomes
\begin{equation}
    \langle \mathcal{H}_1 Y^*, S_{\mathcal{H}} \mathcal{H}_1 Y^* \rangle = \frac{1}{12} \langle \Delta(V_1 Y^*), S_L (V_1 Y^*) \rangle.
\end{equation}
Replacing the remaining Laplacian with $\Delta = -L_0 - 12$, we recall that $L_0 S_L = I$ on the orthogonal complement. Since $V_1 Y^* \sim (Y^*)^2$ contains only even-degree harmonics $l \in \{0, 2, 4, 6\}$, it is orthogonal to $\ker(L_0) = \text{span}\{Y^*\}$ ($l=3$). Thus, $\Delta S_L = -I - 12 S_L$. Therefore, we derive
\begin{align} \label{eq:eta2_term2}
    \langle \mathcal{H}_1 Y^*, S_{\mathcal{H}} \mathcal{H}_1 Y^* \rangle &= \frac{1}{12} \langle V_1 Y^*, \left(-I - 12 S_L\right) V_1 Y^* \rangle \nonumber \\
    &= -\frac{1}{12}\langle V_1^2 Y^*, Y^* \rangle - \langle V_1 Y^*, S_L V_1 Y^* \rangle.
\end{align}
Finally, by subtracting \eqref{eq:eta2_term2} from \eqref{eq:eta2_term1}, the quadratic cross-terms $-\frac{1}{12}\langle V_1^2 Y^*, Y^* \rangle$ cancel, resulting in
\begin{equation}
    \eta_2 = -\langle V_2 Y^*, Y^* \rangle + \langle V_1 Y^*, S_L V_1 Y^* \rangle.
\end{equation}
Comparing this result with \eqref{eq:mu2_app}, we establish the identity
\begin{equation}
    \eta_2 = -\mu_2.
\end{equation}
To determine the formal stability of the bifurcating flow, we evaluate the sign of the perturbed Hessian eigenvalue $\eta(\varepsilon)$ for $|\varepsilon| \ll 1$. Expanding the eigenvalue yields $\eta(\varepsilon) = \eta_0 + \varepsilon \eta_1 + \varepsilon^2 \eta_2 + \mathcal{O}(\varepsilon^3)$. As established above, $\eta_0 = 0$, and $\eta_1 = 0$. Therefore, the expansion collapses to
\begin{equation}
    \eta(\varepsilon) = \varepsilon^2 \eta_2 + \mathcal{O}(\varepsilon^3).
\end{equation}
The sign of the eigenvalue is given by the second-order coefficient, $\text{sign}(\eta(\varepsilon)) = \text{sign}(\eta_2)$. Thus, formal stability $\eta(\varepsilon) < 0$ is achieved if and only if $\eta_2 < 0$.

Having proved that the Hessian and the linearized operator are related by $\eta_2 = -\mu_2$ at the leading order, the stability condition is equivalent to $\mu_2 > 0$. From the Crandall-Rabinowitz bifurcation theorem \cite{crandall1973stability}, the sign of this perturbed eigenvalue is given by 
\begin{equation}
    \text{sign}(\mu_2) = -\text{sign}(\lambda_2 \gamma'(\lambda^*)).
\end{equation}
 Therefore, formal stability is achieved if and only if the bifurcation amplitude $\lambda_2$ and the transversality condition $\gamma'(\lambda^*)$ have opposite signs.
\end{proof}

\section{Stability Analysis of Physical Models}\label{sec:physical_models}

With the general stability criterion established in Theorem \ref{thm:stability}, we evaluate the four profile functions analyzed in our previous bifurcation study \cite{yuri2026}. For each model, we investigate the sign of the product $\lambda_2 \cdot \gamma'(\lambda^*)$ to determine its stability. 

Before proceeding, we recall the core implication of Theorem \ref{thm:stability}: since $\partial_\psi F(\lambda^*, 0) = -12$, the condition $\mu_2 > 0$ ensures that the Energy-Casimir second variation $\delta^2 H_C$
is negative-definite. Consequently, the perturbed flow has a formal maximum of the functional on $X_{\mathbf{T}}$, thus achieving formal stability.

\subsection{Polynomial Model}
We first consider the generalized polynomial nonlinearity \cite{constantin2022,pedlosky1987geophysical},
\begin{equation}
    F(\lambda, \psi) = \mu_1(3\lambda\psi^2 + \psi^3) + (3\mu_1\lambda^2 - [\mu+12])\psi,
\end{equation}
with $\mu, \mu_1 > 0$. Since the potential along the trivial branch evaluates to $\partial_\psi F(\lambda, 0) = 3\mu_1\lambda^2 - \mu - 12$, the eigenvalue is $\gamma(\lambda) = 3\mu_1\lambda^2 - \mu$ \cite{yuri2026}. The bifurcation condition $\gamma(\lambda^*) = 0$ is met at $\lambda^* = \sqrt{\mu/3\mu_1}$, recovering the condition $\partial_\psi F(\lambda^*, 0) = -12$. The derivative of $\gamma$ at the bifurcation point is strictly positive
\begin{equation}
    \gamma'(\lambda^*) = 6\mu_1 \sqrt{\frac{\mu}{3\mu_1}} > 0.
\end{equation}
Furthermore, as established in \cite{yuri2026}, we know that the polynomial model exhibits a \textit{subcritical} bifurcation topology, meaning the amplitude coefficient is strictly negative $\lambda_2 < 0$. Because $\lambda_2$ and $\gamma'(\lambda^*)$ have opposite signs, their product is negative. Applying the Crandall-Rabinowitz criterion from Theorem \ref{thm:stability}, we obtain $\mu_2 > 0$. We therefore conclude that the subcritical tetrahedral flow in this polynomial model is \textbf{formally stable}.

\subsection{Sine-Gordon Model}
We now analyze the sine-Gordon model \cite{flierl1987isolated},
\begin{equation}
    F(\lambda, \psi) = -\lambda \sin(\psi).
\end{equation}
Since the potential along the trivial branch evaluates to $\partial_\psi F(\lambda, 0) = -\lambda \cos(0) = -\lambda$, the eigenvalue is $\gamma(\lambda) = 12 - \lambda$ \cite{yuri2026}. The bifurcation condition $\gamma(\lambda^*) = 0$ is met at the critical threshold $\lambda^* = 12$, recovering the condition $\partial_\psi F(\lambda^*, 0) = -12$ required by Theorem \ref{thm:stability}. 
The derivative of $\gamma$ at the bifurcation point is strictly negative 
\begin{equation}
    \gamma'(\lambda^*) = -1 < 0.
\end{equation}
Furthermore, as established in \cite{yuri2026}, the second-order spatial correction $\psi_2$ vanishes, and the branch bifurcates \textit{supercritically} $\lambda_2 > 0$. Since $\lambda_2 > 0$ and $\gamma'(\lambda^*) < 0$ have opposite signs, their product is negative. Applying the Crandall-Rabinowitz criterion from Theorem \ref{thm:stability}, this gives $\mu_2 > 0$. We therefore conclude that the supercritical tetrahedral flow in the sine-Gordon model is \textbf{formally stable}.

\subsection{Sinh-Gordon Model}
We consider the sinh-Gordon model \cite{stuart1967finite}, governed by 
\begin{equation}
    F(\lambda, \psi) = -\lambda \sinh(\psi).
\end{equation}
Since the potential along the trivial branch evaluates to $\partial_\psi F(\lambda, 0) = -\lambda \cosh(0) = -\lambda$, the eigenvalue is $\gamma(\lambda) = 12 - \lambda$ \cite{yuri2026}, as in the sine-Gordon case. The bifurcation condition $\gamma(\lambda^*) = 0$ is met at the critical threshold $\lambda^* = 12$, recovering the condition $\partial_\psi F(\lambda^*, 0) = -12$. The derivative of $\gamma$ at the bifurcation point is strictly negative
\begin{equation}
    \gamma'(\lambda^*) = -1 < 0.
\end{equation}
Furthermore, as established in \cite{yuri2026}, the sinh-Gordon model exhibits a \textit{subcritical} bifurcation $\lambda_2 < 0$. Since $\lambda_2 < 0$ and $\gamma'(\lambda^*) < 0$ have the same sign, their product is positive. Applying the Crandall-Rabinowitz criterion from Theorem \ref{thm:stability}, this results in $\mu_2 < 0$. We therefore conclude that the subcritical tetrahedral flow in the sinh-Gordon model generates a saddle point and is \textbf{formally unstable}.

\subsection{Exponential Model}

Finally, we examine the Liouville exponential model \cite{joyce1973negative} coupled with a mass constraint,
\begin{equation}
\begin{cases}
    -\Delta\psi + \lambda e^\psi = \frac{\lambda}{4\pi} \iint_{\mathbb{S}^2} e^\psi \, d\sigma,\\
    \iint_{\mathbb{S}^2} \psi \, d\sigma = 0.
\end{cases}
\end{equation}
The linearized operator around a finite-amplitude state $\psi(\varepsilon)$ takes the form
\begin{equation}
    L_{\psi(\varepsilon)}[\delta\psi] = -\Delta\delta\psi + \lambda e^{\psi(\varepsilon)} \delta\psi - \frac{\lambda}{4\pi}\iint_{\mathbb{S}^2} e^{\psi(\varepsilon)} \delta\psi \, d\sigma.
\end{equation}
However, the mass constraint restricts all permissible perturbations $\delta\psi$ and test functions $v$ to the zero-mean functional space. Consequently, when we evaluate $\langle L_{\psi(\varepsilon)}[\delta\psi], v \rangle$, the constant integral term vanishes. We can therefore consider the operator 
\begin{equation}
    \tilde{L} = -\Delta + \lambda e^{\psi(\varepsilon)} I.
\end{equation}
Since the potential along the trivial branch evaluates to $\partial_\psi F(\lambda, 0) = \lambda e^0 = \lambda$, the eigenvalue is $\gamma(\lambda) = 12 + \lambda$ \cite{yuri2026}. The bifurcation condition $\gamma(\lambda^*) = 0$ is met at the critical threshold $\lambda^* = -12$, recovering the condition $\partial_\psi F(\lambda^*, 0) = -12$ required by Theorem \ref{thm:stability}. The derivative of $\gamma$ at the bifurcation point is strictly positive
\begin{equation}
    \gamma'(\lambda^*) = 1 > 0.
\end{equation}
Furthermore, as established in \cite{yuri2026}, this exponential model exhibits a \textit{supercritical} bifurcation $\lambda_2 > 0$. Since $\lambda_2 > 0$ and $\gamma'(\lambda^*) > 0$ have the same sign, their product is positive. Applying the Crandall-Rabinowitz criterion from Theorem \ref{thm:stability}, this results in $\mu_2 < 0$. We therefore conclude that the supercritical tetrahedral flow in the exponential model generates a saddle point and is \textbf{formally unstable}.

\section{Discussion and Conclusions}\label{sec:conclusions}

In this work, we establish an analytical framework to classify the stability of tetrahedral non-zonal flows on the sphere. By proving that the linearized elliptic operator emerging from the Liapunov-Schmidt reduction behaves like the Hessian of the Arnold Energy-Casimir functional, we map the problem of formal stability to the sign of the critical perturbed eigenvalue. 

Our analysis reveals that the definiteness of the second variation is sensitive to the specific nonlinearity of the profile function. For the subcritical sinh-Gordon and supercritical exponential models, the Hessian has a positive eigenvalue, yielding an indefinite second variation. Since definiteness is a necessary condition for nonlinear Lyapunov stability, these flows generate topological saddle points in the Energy-Casimir framework. This property guarantees their linear instability against perturbations.

On the other hand, both the subcritical polynomial model and the supercritical sine-Gordon model achieve a negative-definite Hessian. This property precludes saddle-point formation, providing their formal stability. Therefore, among the considered nonlinearities, only the polynomial and sine-Gordon non-zonal flows are formally stable.

The complete classification of these physical models is summarized in Table \ref{tab:stability_summary}, which highlights the interplay between the bifurcation topology and the spectral crossing rate.

\begin{table}[htbp]
\centering
\renewcommand{\arraystretch}{1.6}
\resizebox{\textwidth}{!}{
\begin{tabular}{@{}l c c c c l@{}}
\toprule
\textbf{Physical Model} & \textbf{Bifurcation} ($\lambda_2$) & \textbf{Spectral Rate} ($\gamma'$) & \textbf{$\lambda_2$, $\gamma'$ Signs} & \textbf{CR Eigenvalue} ($\mu_2$) & \textbf{Formal Stability} \\
\midrule
\textbf{Polynomial} & Subcritical ($-$) & Positive ($+$) & opposite & Positive ($+$) & \textbf{Stable} (Maximum) \\
\textbf{Sine-Gordon} & Supercritical ($+$) & Negative ($-$) & opposite & Positive ($+$) & \textbf{Stable} (Maximum) \\
\textbf{Sinh-Gordon} & Subcritical ($-$) & Negative ($-$) & Concordant & Negative ($-$) & \textbf{Unstable} (Saddle) \\
\textbf{Exponential} & Supercritical ($+$) & Positive ($+$) & Concordant & Negative ($-$) & \textbf{Unstable} (Saddle) \\
\bottomrule
\end{tabular}
}
\vspace{0.2cm}
\caption{Summary of the formal tetrahedral stability analysis. Formal stability in the Arnold Energy-Casimir framework is given by the Crandall-Rabinowitz stability principle. A strictly positive perturbed eigenvalue ($\mu_2 > 0$) guarantees a negative-definite geometric Hessian. This stability is achieved if and only if the topological bifurcation amplitude ($\lambda_2$) and the spectral crossing rate ($\gamma'$) have opposite signs.}
\label{tab:stability_summary}
\end{table}

\textbf{Conflict of interest.} The author declares that he has no conflict of interest.

\textbf{Data availability.} Data sharing is not applicable.
We do not analyse or generate any datasets, because our work proceeds within a theoretical approach.

\bibliographystyle{plain}
\bibliography{bib}

\end{document}